\documentclass[11pt]{article}

\usepackage[a4paper,margin=1in]{geometry}
\usepackage{amsmath,amssymb,amsthm}
\usepackage{mathtools}
\usepackage{graphicx}
\usepackage[hidelinks]{hyperref}
\usepackage{placeins}
\usepackage{float}
\usepackage{enumitem}
\usepackage{booktabs}

\newtheorem{theorem}{Theorem}
\newtheorem{lemma}{Lemma}
\newtheorem{definition}{Definition}
\newtheorem{axiom}{Axiom}
\newtheorem{corollary}{Corollary}
\newtheorem{proposition}{Proposition}

\newcommand{\R}{\mathbb{R}}
\newcommand{\phiG}{\varphi} % golden ratio symbol
\newcommand{\ellE}{\ell}    % edge length
\newcommand{\midpt}{\mathrm{mid}}
\newcommand{\norm}[1]{\left\lVert #1\right\rVert}
\newcommand{\ang}{\angle}

\title{A Ten-Face Non-Edge-Sharing Wing Set on the Regular\\
Icosahedron and a Decagonal Equatorial Balance}
\author{YoungJune Jeon \\ GeoWind}
\date{\today}

\begin{document}
\maketitle

\begin{abstract}
We formalize a ten-face triangular ``wing'' set on a regular icosahedron under a vertex labeling
$\{N,S,U_1,\dots,U_5,L_1,\dots,L_5\}$ with rotation axis $NS$.
The wing faces satisfy: (i) each face is an isosceles $36^\circ\!-\!36^\circ\!-\!108^\circ$ triangle
with a $36^\circ$ angle anchored at a pole ($N$ or $S$); (ii) distinct faces may share vertices but
share no edges; and (iii) a natural equatorial cross-section yields a perfectly balanced regular decagon.
We derive a closed form for the decagon radius,
$R=\frac{\phiG}{2}\ellE$, where $\ellE$ is the icosahedron edge length and
$\phiG=\frac{1+\sqrt5}{2}$ is the golden ratio.
Beyond the geometric results, we interpret the ``ten''-face closure as a symmetry-consistent design principle
for a pole-anchored wing layout and provide a reproducible construction workflow.
\end{abstract}

\section{Introduction}
Regular polyhedra provide highly constrained geometric structures where design rules can be stated
as axioms and verified as theorems.
This paper defines and analyzes a specific triangular face set motivated by a GeoWind wing rule-set.
Our first goal is mathematical: a precise definition and closed-form consequences for angles, symmetry,
and a cross-sectional balance.
Our second goal is design-facing: to express the same structure as a reproducible, checkable rule-set
suitable for implementation in CAD/parametric geometry.

At a high level, the construction selects ten ``golden'' isosceles triangles (angles $36^\circ$--$36^\circ$--$108^\circ$)
whose small angle is anchored at one of two poles on the rotation axis, and whose edges are pairwise disjoint
(edge non-sharing). When all ten faces are present, their pole-opposite edges induce a regular decagon on the
equatorial plane with a radius determined by the golden ratio.

\section{Preliminaries}
Let $\phiG=\frac{1+\sqrt5}{2}$ denote the golden ratio, satisfying $\phiG^2=\phiG+1$.
Let $I$ be a regular icosahedron with edge length $\ellE$ and center $O$.
Choose opposite vertices $N$ and $S$ and call the line $NS$ the rotation axis.

\subsection{A standard coordinate model}
For proofs that require explicit distance computations, we use the standard coordinate model of the icosahedron
with edge length $2$ whose vertices are
\begin{equation}\label{eq:std-icosa}
(0,\pm1,\pm\phiG),\quad (\pm1,\pm\phiG,0),\quad (\pm\phiG,0,\pm1).
\end{equation}
Scaling by $\ellE/2$ converts this model to edge length $\ellE$.

\section{Labeling and Wing-Face Axioms}

\begin{definition}[Vertex labeling]\label{def:label}
Label the 12 vertices of the icosahedron as
\[
V=\{N,S,U_1,\dots,U_5,L_1,\dots,L_5\},
\]
where indices are taken modulo $5$.
The axis is the line $NS$.
\end{definition}

Figure~\ref{fig:labels} shows the labeled icosahedron with the rotation axis $NS$ and the two vertex rings.

\begin{figure}[H]
  \centering
  \includegraphics[width=0.6\linewidth]{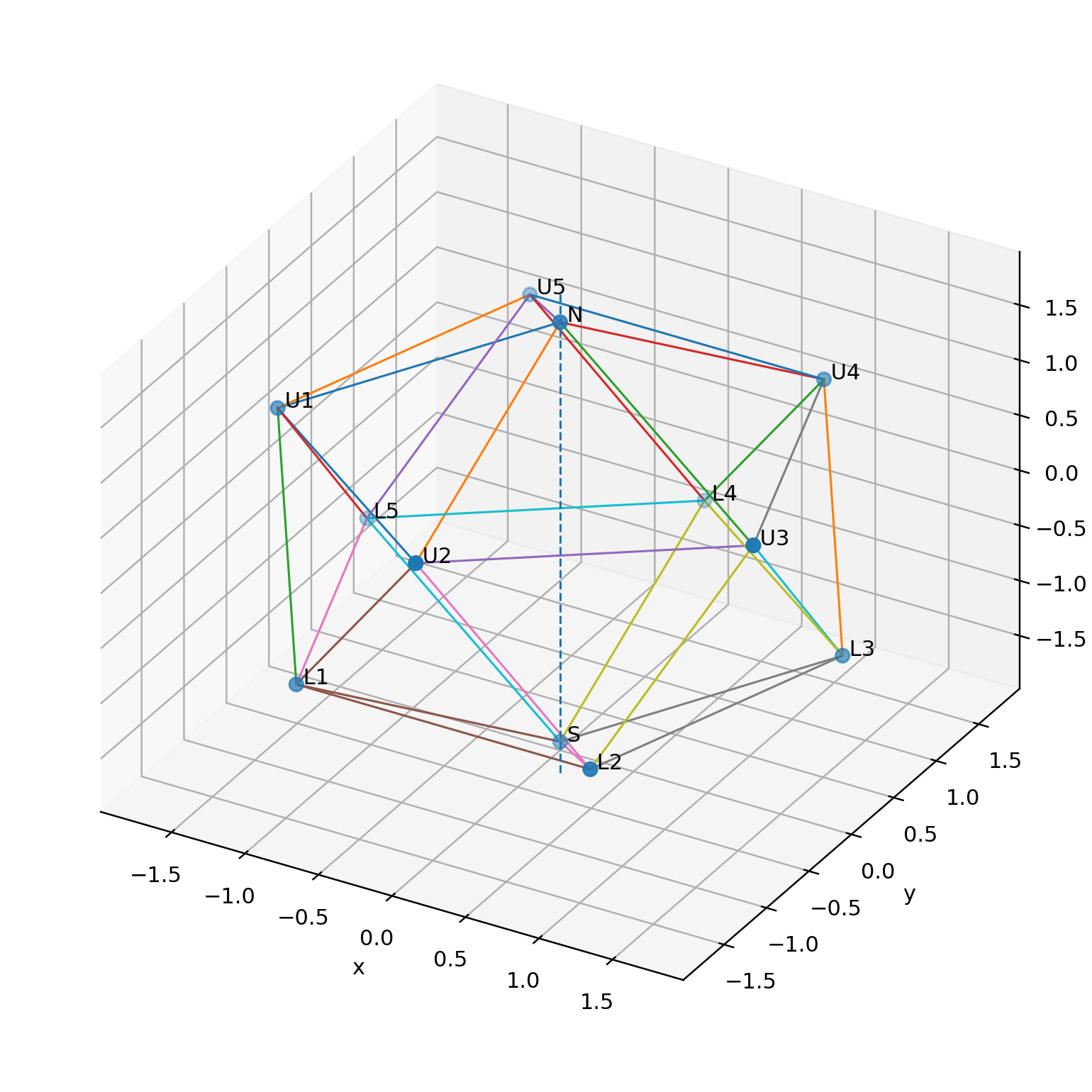}
  \caption{Labeled regular icosahedron used throughout the paper. The rotation axis is the line
  $NS$ (aligned with the $z$-axis). The upper-ring vertices are labeled $U_1,\dots,U_5$ and the lower-ring
  vertices $L_1,\dots,L_5$.}
  \label{fig:labels}
\end{figure}

\begin{definition}[GeoWind ten-face wing set]\label{def:wingset}
Define the wing-face set $\mathcal{F}$ as the following ten triangles:
\begin{align*}
F_S(i) &= \triangle(S,U_i,L_i), && i=1,\dots,5,\\
F_N(i) &= \triangle(N,U_i,L_{i-1}), && i=1,\dots,5,
\end{align*}
with $L_0:=L_5$.
\end{definition}

\begin{axiom}[Edge non-sharing; vertex sharing allowed]\label{ax:edge}
For any two distinct faces $F_a\neq F_b\in\mathcal{F}$,
\[
E(F_a)\cap E(F_b)=\varnothing,
\]
i.e., no edge is shared between two faces. Vertex sharing is allowed.
\end{axiom}

\begin{axiom}[Non-intersection (geometric layout)]\label{ax:nonintersect}
The interiors of distinct faces do not intersect:
for $F_a\neq F_b\in\mathcal{F}$,
\[
\mathrm{int}(F_a)\cap \mathrm{int}(F_b)=\varnothing.
\]
\end{axiom}

The ten-face wing set $\mathcal{F}$ embedded in the labeled icosahedron is illustrated in Figure~\ref{fig:tenfaces}.

\begin{figure}[H]
  \centering
  \includegraphics[width=0.6\linewidth]{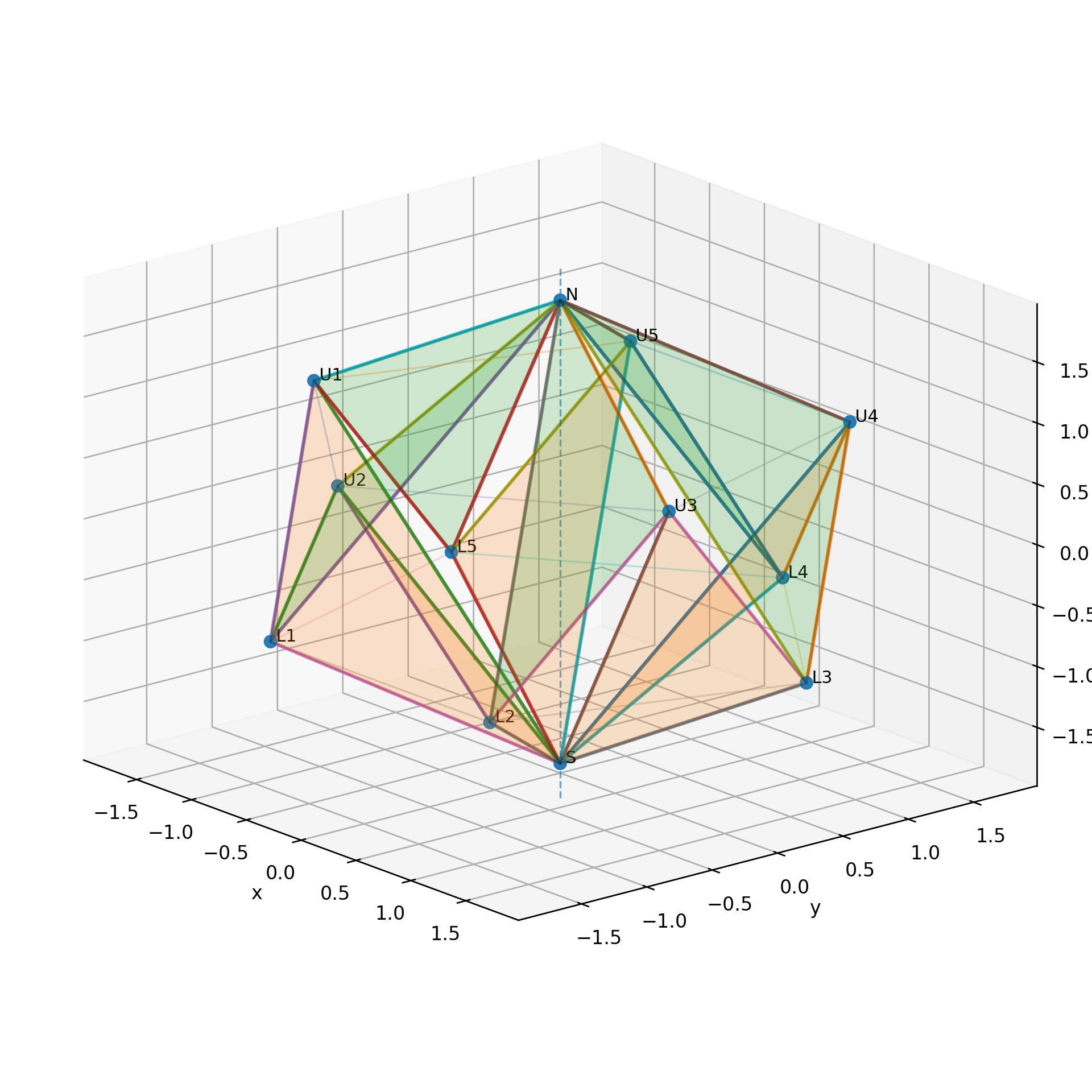}
  \caption{The GeoWind ten-face wing set $\mathcal{F}$ embedded in the labeled icosahedron.
  South-anchored faces are $F_S(i)=\triangle(S,U_i,L_i)$ and north-anchored faces are
  $F_N(i)=\triangle(N,U_i,L_{i-1})$ (indices modulo 5).
  Faces are shown translucent; vertex sharing is allowed, but no two faces share an edge.}
  \label{fig:tenfaces}
\end{figure}

\FloatBarrier

\section{Main Results}

\subsection{Angle structure via a golden-distance lemma}

\begin{lemma}[Edge and diagonal distances in the standard model]\label{lem:dist}
In the standard coordinate model \eqref{eq:std-icosa} (edge length $2$),
(i) adjacent vertices have distance $2$, and
(ii) for an appropriate choice of opposite poles $N,S$ and a corresponding upper vertex $U$,
the distance between a pole and an adjacent upper vertex is $2$, whereas the distance between the opposite pole and that upper vertex equals $2\phiG$.
Equivalently, after scaling to edge length $\ellE$, the corresponding distances are $\ellE$ and $\phiG \ellE$.
\end{lemma}

\begin{proof}
For (i), one may verify adjacency by direct computation for a representative pair.
Let $A=(0,1,\phiG)$ and $B=(1,\phiG,0)$; then
\[
\norm{A-B}^2=(0-1)^2+(1-\phiG)^2+(\phiG-0)^2
=1+(1-2\phiG+\phiG^2)+\phiG^2.
\]
Using $\phiG^2=\phiG+1$, the expression becomes
\[
1+(1-2\phiG+\phiG+1)+(\phiG+1)=4,
\]
hence $\norm{A-B}=2$. Similar checks generate the full edge set by symmetry.

For (ii), take poles $N=(0,1,\phiG)$ and $S=(0,-1,-\phiG)$ (these are opposite in \eqref{eq:std-icosa}).
Let $U=(1,\phiG,0)$, which is adjacent to $N$ by (i). Then
\[
\norm{N-U}=2,\qquad
\norm{S-U}^2=(0-1)^2+(-1-\phiG)^2+(-\phiG-0)^2.
\]
Compute
\[
\norm{S-U}^2=1+(1+2\phiG+\phiG^2)+\phiG^2
=2+2\phiG+2\phiG^2.
\]
Substituting $\phiG^2=\phiG+1$ gives
\[
\norm{S-U}^2=2+2\phiG+2(\phiG+1)=4+4\phiG=4\phiG^2,
\]
so $\norm{S-U}=2\phiG$.
Scaling the coordinate model by factor $\ellE/2$ scales all distances by $\ellE/2$, yielding $\ellE$ and $\phiG\ellE$.
\end{proof}

\begin{lemma}[Cosine-law characterization of the golden gnomon]\label{lem:gnomon}
Let a triangle have side lengths $(\ellE,\ellE,\phiG\ellE)$. Then its angles are $(36^\circ,36^\circ,108^\circ)$.
\end{lemma}

\begin{proof}
Let the equal sides have length $\ellE$, and let the base have length $\phiG\ellE$.
By the law of cosines, the angle $\theta$ opposite the base satisfies
\[
\cos\theta=\frac{\ellE^2+\ellE^2-(\phiG\ellE)^2}{2\ellE\cdot\ellE}=\frac{2-\phiG^2}{2}.
\]
Using $\phiG^2=\phiG+1$ gives
\[
\cos\theta=\frac{2-(\phiG+1)}{2}=\frac{1-\phiG}{2}=-\cos 72^\circ,
\]
hence $\theta=108^\circ$. The remaining two angles are equal and sum to $72^\circ$, so each is $36^\circ$.
\end{proof}

\begin{theorem}[Golden-triangle shape: $36^\circ\!-\!36^\circ\!-\!108^\circ$]\label{thm:golden}
Every face $F\in\mathcal{F}$ is an isosceles triangle with interior angles $(36^\circ,36^\circ,108^\circ)$.
Moreover, for $F_S(i)$ a $36^\circ$ angle occurs at the pole $S$, and for $F_N(i)$ a $36^\circ$ angle occurs at $N$.
\end{theorem}

\begin{proof}
Fix $i$ and consider $F_S(i)=\triangle(S,U_i,L_i)$.
In the regular icosahedron, $S$ is adjacent to $L_i$ and $U_i$ is adjacent to $L_i$, so
\[
|SL_i|=\ellE,\qquad |U_iL_i|=\ellE.
\]
By Lemma~\ref{lem:dist}, the distance from $S$ to the corresponding upper vertex is $\phiG\ellE$, so
\[
|SU_i|=\phiG\ellE.
\]
Thus $F_S(i)$ has side lengths $(\ellE,\ellE,\phiG\ellE)$, and by Lemma~\ref{lem:gnomon} its angles are
$(36^\circ,36^\circ,108^\circ)$ with a $36^\circ$ angle at $S$.

For $F_N(i)=\triangle(N,U_i,L_{i-1})$, $N$ is adjacent to $U_i$ and $U_i$ is adjacent to $L_{i-1}$, hence
\[
|NU_i|=\ellE,\qquad |U_iL_{i-1}|=\ellE.
\]
Again by Lemma~\ref{lem:dist}, the remaining side has length $\phiG\ellE$, so the same angle conclusion follows,
with a $36^\circ$ angle anchored at $N$.
\end{proof}

\subsection{Edge non-sharing and maximality}

\begin{theorem}[No shared edges within the ten faces]\label{thm:noedge}
The ten faces in $\mathcal{F}$ satisfy Axiom~\ref{ax:edge}: no two distinct faces share an edge.
\end{theorem}

\begin{proof}
Each south face $F_S(i)$ has edge set $\{SU_i,\,SL_i,\,U_iL_i\}$.
For $i\neq j$ these edges have different endpoints and thus cannot coincide.
Each north face $F_N(i)$ has edge set $\{NU_i,\,NL_{i-1},\,U_iL_{i-1}\}$, and the same endpoint argument applies.

It remains to compare a south edge with a north edge.
Any south edge incident to $S$ cannot equal a north edge incident to $N$.
A south edge of type $U_iL_i$ cannot equal a north edge of type $U_jL_{j-1}$ because equality would require
$U_i=U_j$ and $L_i=L_{j-1}$, hence $i=j=j-1 \ (\mathrm{mod}\ 5)$, impossible.
Therefore no edges are shared.
\end{proof}

\begin{theorem}[Maximality under pole-anchored edge non-sharing]\label{thm:max}
Consider pole-anchored wing triangles of the form $\triangle(S, U_i, \cdot)$ whose edges must be pairwise non-shared.
Then at most five such triangles can be anchored at $S$. Likewise, at most five can be anchored at $N$.
Consequently, under pole anchoring at both poles and edge non-sharing, a construction can contain at most ten faces.
\end{theorem}

\begin{proof}
Any triangle anchored at $S$ using an upper vertex $U_i$ necessarily contains the edge $SU_i$.
Under edge non-sharing, no two such triangles may use the same edge $SU_i$.
Since there are only five distinct upper vertices $U_1,\dots,U_5$, there are only five distinct edges $SU_i$,
hence at most five edge-disjoint pole-anchored triangles can be anchored at $S$.
The same argument holds for $N$.
\end{proof}

\subsection{Equatorial decagon and closed-form radius}

\begin{definition}[Equatorial plane and cross-edge midpoints]\label{def:equator}
Let the equatorial plane be
\[
\Pi:=\{x\in\R^3:(x-O)\cdot (N-S)=0\},
\]
i.e., the plane through $O$ perpendicular to axis $NS$.
For a face $F$, let $c(F)$ denote the edge not incident to the corresponding pole (the ``cross-edge'').
Define the representative point
\[
p(F):=\midpt(c(F)),
\]
the midpoint of the cross-edge.
\end{definition}

\begin{proposition}[Cross-edges for the ten faces]\label{prop:cross}
For the faces in Definition~\ref{def:wingset},
\[
c(F_S(i))=U_iL_i,\qquad c(F_N(i))=U_iL_{i-1}.
\]
\end{proposition}

\begin{proof}
In $F_S(i)=\triangle(S,U_i,L_i)$ the two edges incident to pole $S$ are $SU_i$ and $SL_i$, hence the remaining edge is $U_iL_i$.
The north case is identical.
\end{proof}

\begin{theorem}[Perfect circular balance: regular decagon on $\Pi$]\label{thm:decagon}
The set $\{p(F)\mid F\in\mathcal{F}\}\subset\Pi$ consists of ten points lying on a single circle centered at $O$,
forming a regular decagon with $36^\circ$ angular spacing.
\end{theorem}

\begin{proof}
By Proposition~\ref{prop:cross}, each $p(F)$ is the midpoint of a $U$--$L$ edge.
By symmetry about the axis and the central plane through $O$ orthogonal to $NS$, each such midpoint lies in $\Pi$,
so $p(F)\in\Pi$.

Next, consider the two midpoint sets
\[
P_1=\{\midpt(U_iL_i)\}_{i=1}^5,\qquad P_2=\{\midpt(U_iL_{i-1})\}_{i=1}^5.
\]
Rotation by $72^\circ$ about axis $NS$ maps $U_i\mapsto U_{i+1}$ and $L_i\mapsto L_{i+1}$,
so it cyclically permutes $P_1$ and also $P_2$. Hence each of $P_1$ and $P_2$ forms a regular pentagon on a circle
centered at $O$ in $\Pi$.

To show that $P_1\cup P_2$ is a regular decagon, it suffices to show a $36^\circ$ phase offset between one point in $P_1$
and a neighboring point in $P_2$. This is established in Proposition~\ref{prop:phase36}. Therefore the two pentagons
interlace into a regular decagon with $36^\circ$ spacing.
\end{proof}

\begin{proposition}[Phase offset: $\cos 36^\circ=\phiG/2$]\label{prop:phase36}
In the standard model \eqref{eq:std-icosa} (edge length $2$), let
\[
M=\midpt\big((0,1,\phiG),(1,\phiG,0)\big),\qquad
M'=\midpt\big((0,1,\phiG),(-1,\phiG,0)\big).
\]
Then $\ang MOM'=36^\circ$ and
\[
\frac{M\cdot M'}{\norm{M}\,\norm{M'}}=\frac{\phiG}{2}.
\]
\end{proposition}

\begin{proof}
Compute
\[
M=\left(\frac12,\frac{1+\phiG}{2},\frac{\phiG}{2}\right),\qquad
M'=\left(-\frac12,\frac{1+\phiG}{2},\frac{\phiG}{2}\right).
\]
Their dot product is
\[
M\cdot M'=-\frac14+\left(\frac{1+\phiG}{2}\right)^2+\left(\frac{\phiG}{2}\right)^2.
\]
Using $\phiG^2=\phiG+1$,
\[
\left(\frac{1+\phiG}{2}\right)^2=\frac{1+2\phiG+\phiG^2}{4}=\frac{2+3\phiG}{4},\qquad
\left(\frac{\phiG}{2}\right)^2=\frac{\phiG^2}{4}=\frac{\phiG+1}{4}.
\]
Therefore
\[
M\cdot M'=-\frac14+\frac{2+3\phiG}{4}+\frac{\phiG+1}{4}=\frac{2+4\phiG}{4}=\frac{1+2\phiG}{2}.
\]
Meanwhile, $\norm{M}=\norm{M'}=\phiG$ (see Theorem~\ref{thm:radius}), so $\norm{M}\norm{M'}=\phiG^2$.
Thus
\[
\frac{M\cdot M'}{\norm{M}\norm{M'}}=\frac{(1+2\phiG)/2}{\phiG^2}=\frac{\phiG}{2},
\]
which equals $\cos 36^\circ$. Hence $\ang MOM'=36^\circ$.
\end{proof}

\begin{theorem}[Closed-form radius $R=\frac{\phiG}{2}\ellE$]\label{thm:radius}
Let $R$ be the radius of the circle in Theorem~\ref{thm:decagon}. Then
\[
\boxed{R=\frac{\phiG}{2}\,\ellE}.
\]
\end{theorem}

\begin{proof}
In the standard model \eqref{eq:std-icosa} (edge length $2$), consider adjacent vertices
$A=(0,1,\phiG)$ and $B=(1,\phiG,0)$. Their midpoint is
\[
M=\left(\frac12,\ \frac{1+\phiG}{2},\ \frac{\phiG}{2}\right).
\]
Using $\phiG^2=\phiG+1$,
\[
\norm{M}^2=\left(\frac12\right)^2+\left(\frac{1+\phiG}{2}\right)^2+\left(\frac{\phiG}{2}\right)^2=\phiG^2,
\]
so $\norm{M}=\phiG$. Hence the midpoint circle in the edge-length-$2$ model has radius $\phiG$.
Scaling to edge length $\ellE$ multiplies distances by $\ellE/2$, giving $R=\phiG(\ellE/2)=\frac{\phiG}{2}\ellE$.
\end{proof}

\begin{corollary}
If $\ellE=1\,\mathrm{m}$, then
\[
R=\frac{\phiG}{2}\ \mathrm{m}\approx 0.809016994\ \mathrm{m}.
\]
\end{corollary}

Figure~\ref{fig:decagon} illustrates the equatorial plane $\Pi$ and the regular decagon formed by the
representative points $p(F)$.

\begin{figure}[H]
  \centering
  \includegraphics[width=0.6\linewidth]{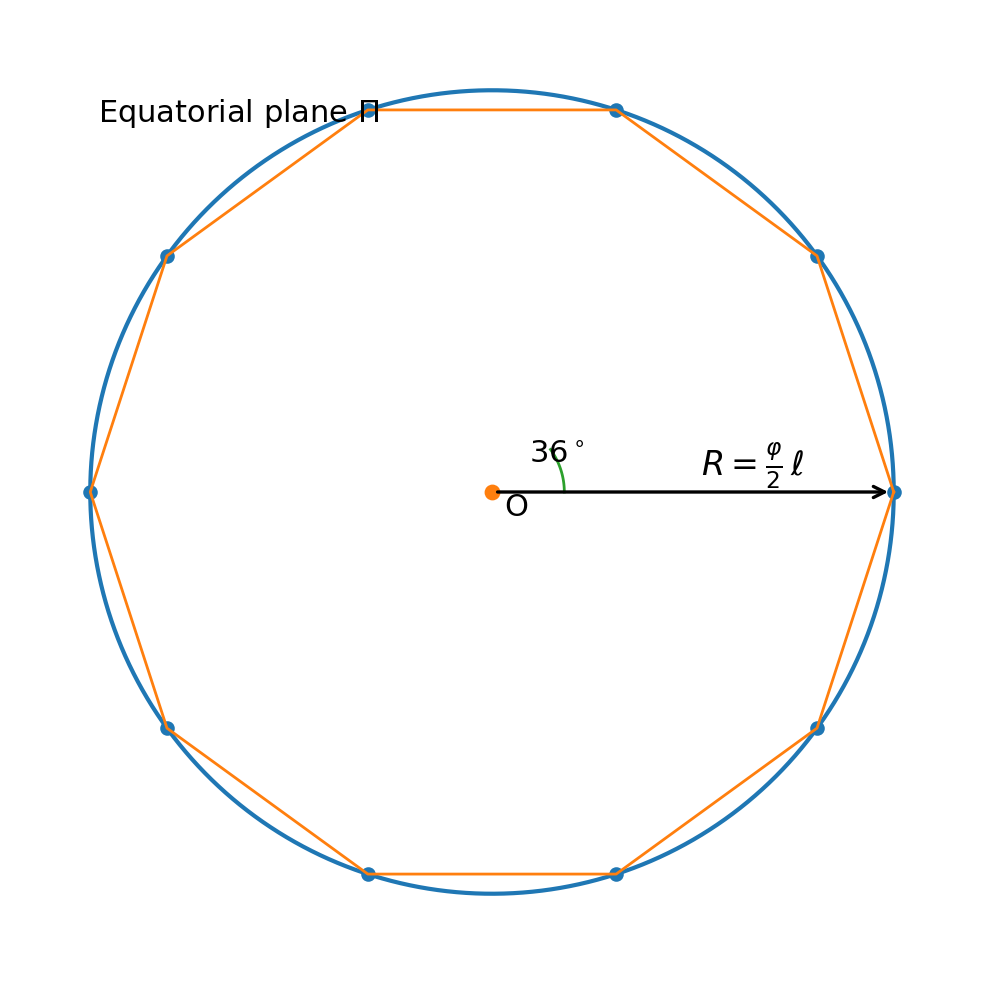}
  \caption{Equatorial plane $\Pi$ (perpendicular to axis $NS$ through center $O$).
  The ten representative points $p(F)$ (midpoints of the pole-opposite edges) lie on a circle
  centered at $O$ and form a regular decagon of radius $R=\frac{\phiG}{2}\ellE$.}
  \label{fig:decagon}
\end{figure}

\FloatBarrier

\section{Design Interpretation: Why ``Ten''?}
The number ten is not merely a count; it encodes a closure condition consistent with the icosahedron's
fivefold rotational symmetry about the pole axis.
On each pole, there are exactly five distinct edges $SU_i$ (and $NU_i$). Under edge non-sharing, these edges
act as five ``ports'' for pole-anchored faces (Theorem~\ref{thm:max}).
The chosen index shift between south and north faces prevents edge duplication while interlacing the
two pentagonal midpoint sets into a decagon (Theorem~\ref{thm:decagon}).

From a design standpoint, the equatorial decagon provides a compact diagnostic:
if a CAD model implements the ten faces correctly, the midpoints of the pole-opposite edges must lie on a
single circle and fall into a regular decagon with radius $\frac{\phiG}{2}\ellE$.
This offers a simple geometric ``sanity check'' independent of aerodynamic modeling.

\section{Reproducible Construction Workflow (CAD/Parametric)}
We summarize a practical procedure to generate and validate $\mathcal{F}$.

\subsection*{Algorithm 1 (Wing-face generation and validation)}
\begin{enumerate}[leftmargin=1.8em]
\item \textbf{Create a labeled icosahedron.} Fix edge length $\ellE$ and label vertices as in Definition~\ref{def:label}.
\item \textbf{Generate faces.} Create $F_S(i)=\triangle(S,U_i,L_i)$ and $F_N(i)=\triangle(N,U_i,L_{i-1})$ for $i=1,\dots,5$.
\item \textbf{Edge non-sharing check.} List all edges of the ten triangles (unordered vertex pairs) and verify no duplicates.
\item \textbf{Angle check.} Verify side lengths $(\ellE,\ellE,\phiG\ellE)$ and angles $(36^\circ,36^\circ,108^\circ)$ for each face.
\item \textbf{Equatorial decagon check.} Compute $p(F)$ for all faces. Confirm decagon geometry and radius $R=\frac{\phiG}{2}\ellE$.
\item \textbf{Non-intersection check (optional).} Perform triangle--triangle intersection tests to numerically verify Axiom~\ref{ax:nonintersect}.
\end{enumerate}

\section{Related Geometric Context}
The appearance of $\phiG$ in the icosahedron is classical: coordinate realizations, vertex/edge relations,
and duality with the dodecahedron all involve the golden ratio.
The present contribution is a specific edge-disjoint, pole-anchored ten-face rule-set whose equatorial midpoints
close into a decagon with a simple closed-form radius.
For general background, see Coxeter and Gr\"unbaum.

\section{Conclusion}
We defined a pole-anchored ten-face wing set on a regular icosahedron and proved:
(i) each face is a $36^\circ\!-\!36^\circ\!-\!108^\circ$ isosceles (golden) triangle;
(ii) faces share no edges and at most ten such pole-anchored edge-disjoint faces can exist across both poles; and
(iii) the pole-opposite edge midpoints form a regular decagon on the equatorial plane with radius $R=\frac{\phiG}{2}\ellE$.
These results provide a mathematical foundation and a practical validation workflow for geometry-driven design implementations.

\bibliographystyle{plain}

\end{document}